\documentclass[12pt,a4paper]{amsart}
\usepackage{paralist, mathdots} 

\usepackage{amssymb} 
\usepackage{amsmath,tikz,wasysym,multirow} 

\newtheorem{proposition}{Proposition}[section]
\newtheorem{theorem}{Theorem}[section]
\newtheorem{example}{Example}[section]

\newtheorem{lemma}{Lemma}[section]

\newcommand{\npmatrix}[1]{\left( \begin{matrix} #1 \end{matrix} \right)}

    \definecolor{helena}{rgb}{.2,.8,.4}
    \definecolor{polona}{rgb}{.8,.2,.2}
   \definecolor{todo}{rgb}{.2,.2,.8}

\begin{document}
\title{The integer cp-rank of $2 \times 2$ matrices}
\author{Thomas Laffey, Helena \v Smigoc}
%
\address{H.\v Smigoc:~School of Mathematical Sciences, University College Dublin, Belfield, Dublin 4, Ireland; email: helena.smigoc@ucd.ie}
\address{T.J.~Laffey:~School of Mathematical Sciences, University College Dublin, Belfield, Dublin 4, Ireland; email: thomas.laffey@ucd.ie}
%
\bigskip

\maketitle

\begin{abstract}
We show the cp-rank of an integer doubly nonnegative $2 \times 2$ matrix does not exceed $11$. 

\noindent\emph{Key words:} completely positive matrices, doubly nonnegative matrices, integer matrices. 
	MSC[2010] 15B36, 	15B48.
\end{abstract}

\quad

\emph{Dedicated to Charles R. Johnson. }

\section{Introduction}

A $n\times n$ matrix $A$ is said to be \emph{completely positive}, if there exists a (not necessarily square) nonnegative matrix $V$ such that $A=VV^T$. 
Completely positive matrices have been widely studied, and they play an important role in various applications. It is a subject to which C.R. Johnson has made an important contribution \cite{MR1823345,MR1770363, MR1638954,MR1395610}.  For further background on completely positive matrices, we refer the reader to the following works and citations therein \cite{MR1986666,MR3414584,MR2892529, MR2845851}.

Clearly, any completely positive matrix is nonnegative and positive semidefinite. We call the family of matrices that are both nonnegative and positive semidefinite \emph{doubly nonnegative}. Doubly nonnegative matrices of
order less than $5$ are completely positive \cite{MR0174570}. However, this is no longer true for matrices of order larger than or equal to $6$ \cite{MR0111694}.

Any $n \times n$ completely positive matrix $A$ has many \emph{cp-factorizations} of the form $A=VV^T,$ where $V$ is an $n \times m$ matrix. Note that $m$ is also not unique. We define the \emph{cp-rank} of $A$ to be the minimal possible $m$. If we demand that $V$ has rational entries, then we say that $A$ has \emph{a rational cp-factorization}. We define the \emph{rational cp-rank} correspondingly. In this note we will study \emph{integer cp-factorizations}, where we demand $V$ to be an integer nonnegative matrix, and \emph{integer cp-rank}, the minimal number of columns in the integer cp-factorization of a given matrix.

Every rational matrix which lies in the interior of the cone of completely positive matrices has a rational cp-factorization \cite{MR3624664}, but the question is still open for rational matrices on the boundary of the region. On the other hand, for $n \geq 3$ it it easy to find examples of $n \times n$ integer completely positive matrices that do not have an integer cp-factorization. In \cite{MR3904097}  the authors answered a question posed in \cite{Berman2}, by proving that for $n=2$ every integer doubly nonnegative matrix has an integer cp-factorization. An alternative proof of this result can be found in \cite{sikiri2018simplex}.  Neither of those proofs offer a bound on the integer cp-rank of such matrices. In this note we prove that the integer cp-rank of $2 \times 2$ matrices cannot be larger than $11$.

\section{Main Result}

The question of determining the completely positive integer rank for a given $n\times n$ completely positive matrix is not trivial, even in the case when $n=1$. In this case the answer is given by Lagrange's Four-Square Theorem. 

\begin{theorem}\label{thm:foursquares}[Lagrange's Four Square Theorem]
Every positive integer $x$ can be written as the sum of at most four squares. If $x$ is not of the form 
\begin{equation}\label{eq:four}
x=4^r(8k+7)
\end{equation}
 for some nonnegative integers $r$ and $k$, then $x$ is the sum of at most three squares. 
\end{theorem}

With Theorem \ref{thm:foursquares}, rank one  matrices are easy to analyse. 

\begin{lemma}\label{lem:rank1}
An integer doubly nonnegative matrix $A$ of rank $1$ is completely positive, and has the integer cp-rank equal to the integer cp-rank of the greatest common divisor of its diagonal elements. 
\end{lemma}

\begin{proof}
Let $A=(a_{ij})$ be an integer doubly nonnegative matrix of rank $1$, and let $d:=\gcd(a_{11},a_{22},\ldots,a_{nn})$. Since $a_{ij}=\sqrt{a_{ii}a_{jj}}$, $d$ also divides all the off-diagonal elements of $A$. Hence, $A=dB$, where $B=(b_{ij})$ satisfies $\gcd(b_{11},b_{22},\ldots,b_{nn})=1$.

We claim that each diagonal element in $B$ is a perfect square. If this is not true, then $b_{i_0i_0}=p c_{i_0}^2$, where $p$ is a product of distinct primes. Now $b_{i_0j}^2=b_{i_0i_0}b_{jj}=pc_{i_0}^2 b_{jj}$. Hence $b_{jj}=p c_j^2$ for some positive integer $c_j$ and for all $j$. This implies that $p$ divides $b_{jj}$ for all $j$. As this contradicts our assumption, the claim is proved. 
\end{proof}

Proposition below not only gives the first bound on the integer cp-rank of $2 \times 2$ matrices, but it also provides an approach that is later refined to improve the bound.

\begin{proposition}
A $2 \times 2$ integer doubly nonnegative matrix has an integer cp-factorization, and an integer cp-rank less than or equal to $12$. 
\end{proposition}

\begin{proof} Let $$A=\npmatrix{a & b \\ b& c}$$ be an integer doubly nonnegative matrix. 
First we prove that we can reduce our problem to the case when $a \geq b$ and $c \geq b$. To this end we assume $b>c$, and write $b=\alpha c+b_0$, where $\alpha \geq 1$ is a positive integer, and $b_0 \in \{0,\ldots, c-1\}$. Let $$S(\alpha):=\npmatrix{1 & -\alpha \\ 0& 1}.$$ We claim that 
$$A_0=S(\alpha)AS(\alpha)^T=\npmatrix{a-2 \alpha b +\alpha^2 c & b-\alpha c \\ b-\alpha c & c}=\npmatrix{a_0 & b_0 \\ b_0 & c_0}$$ 
is a doubly nonnegative matrix. From $\det S(\alpha)=1$, we deduce $\det A_0=\det A \geq 0$. Now,  the inequalities
$\det A= a_0 c_0-b_0^2\geq 0$ and $c_0=c >0$, imply that $a_0=(a-2 \alpha b_0-\alpha c)\geq 0$. 

Since $S(\alpha)^{-1}>0$, any completely positive factorization of $A_0$:  $A_0=B_0B_0^T$, gives us a completely positive factorization of $A$: $$A=(S(\alpha)^{-1}B_0)(S(\alpha)^{-1}B_0)^T.$$ Clearly, $B_0$ and $(S(\alpha)^{-1}B_0)$ have the same number of columns, hence the two factorizations give the same bound on the cp-rank. With this we have proved, that to find an integer completely positive factorization for $A$ it is sufficient to solve the problem for $A_0$, that satisfies $a_0 \geq b_0$ and $c_0\geq b_0$. If $b>a$, we can repeat the above argument, with the roles of the diagonal elements reversed. 

 From now on we may assume that our given matrix $A$ satisfies $a \geq b$ and $c \geq b$.  Under this assumption we can write: 
\begin{align}\label{eq:bound12}
A&=\npmatrix{b & b \\ b & b}+\npmatrix{a-b & 0 \\ 0 & c-b}\\
&=b\npmatrix{1 & 1 \\ 1 & 1}+\npmatrix{a-b & 0 \\ 0 & 0}+\npmatrix{0 & 0 \\ 0 & c-b}. 
\end{align}
By Lemma \ref{lem:rank1} each of the rank $1$ matrices in the above sum have the integer cp-rank at most $4$, so the integer cp-rank of $A$ is at most $12$. 
\end{proof}

To reduce the bound for the cp-rank to $11$ we look more closely at the family of integers that cannot be written as a sum of less than four squares. 

\begin{lemma}\label{lem:three}
Let $x$ be a positive integer of the form \eqref{eq:four}. Then $x-2$, $x-6$, $x+2$ and $x+6$ are not of the form  \eqref{eq:four}. 
\end{lemma}

\begin{proof}
 Let $x=4^r(8k+7)$ for some nonnegative integers $r$ and $k$. Then: 
\begin{align*}
x &\equiv 7 \pmod{8} \text{ when } r=0, \\
x &\equiv 4 \pmod{8} \text{ when } r=1, \\
x &\equiv 0 \pmod{8} \text{ when } r\geq 2.  \\
\end{align*}
In each case, it is straightforward to check that $x-6$, $x-2$, $x+2$, $x+6$ are not equivalent to $7$, $4$ or $0$ modulo $8$, so they cannot be of the form \eqref{eq:four}. 
\end{proof}

\begin{theorem}
Let $$A=\npmatrix{a & b \\ b& c}$$ be an integer doubly nonnegative matrix. Then $A$ has an integer cp-factorization, and an integer cp-rank less than or equal to $11$. 
\end{theorem}
  
\begin{proof}
 From \eqref{eq:bound12} and Theorem \ref{thm:foursquares} it is clear that the bound $12$ will not be reached unless $b$, $a-b$ and $c-b$ are all of the form  \eqref{eq:four}. In particular, the result holds for $b \leq 6$, and for $a-b \leq 6$ or $c-b \leq 6$. So we assume $b$, $a-b$ and $c-b$ are all greater than or equal to $7$, and that they all require four squares in Theorem \ref{thm:foursquares}.

First let us consider the case when $a-b \not\equiv 7 \pmod{8}$. In this case $a-b-3\equiv 1 \pmod{8}$ or $a-b-3\equiv 5 \pmod{8}$, so $a-b-3$ is not of the form \eqref{eq:four}. We write: 
 \begin{align*}
 A=\npmatrix{a-b-3 & 0 \\ 0 & c-b+2}+(b-6) \npmatrix{1 & 1 \\ 1 & 1}+\npmatrix{3 \\ 2} \npmatrix{3 & 2}. 
 \end{align*}
Under our assumption, we can write each $a-b-3$, $c-b+2$ and $b-6$ as sums of at most three squares by Lemma \ref{lem:three}, so the integer completely positive rank of $A$ is at most $3+3+3+1=10$. The case, when  $c-b \not\equiv 7 \pmod{8}$ can be dealt with in a similar way. 
 
 Now we assume that $a-b \equiv 7 \pmod{8}$ and $c-b \equiv 7 \pmod{8}$. We write: 
 \begin{align*}
 A=\npmatrix{a-b-1 & 0 \\ 0 & c-b+2}+(b-2) \npmatrix{1 & 1 \\ 1 & 1}+\npmatrix{1 \\ 2} \npmatrix{1 & 2}. 
 \end{align*}
 Since $a-b-1 \equiv 6 \pmod{8}$ it is not of the form \eqref{eq:four}, and $c-b+2$ and $b-2$ are not of the form \eqref{eq:four} by Lemma \ref{lem:three}, the integer completely positive rank of $A$ is at most $3+3+3+1=10$. 
\end{proof}  

Next example shows that the integer cp-rank of a $2 \times 2$ matrix can be as high as $9$, but we were not able to find examples of $2 \times 2 $ matrices with cp-rank larger than that. 

\begin{example}
Let $$A=\npmatrix{a & 1 \\ 1 & c},$$ where $a$ and $c$ are positive integers. Then any integer cp-factorization of $A$ must involve $\npmatrix{1 \\ 1} \npmatrix{1 & 1}$, so the integer cp-rank of $A$ is $1+p+q$, where $p$, $q$ are the least number of squares of integers needed to represent $a-1$, $c-1$, respectively. In particular, $A$ has the integer cp-rank $9$ if $a$ and $c$ are both divisible by $8$.    
\end{example}
    
\begin{example}
Let $$B=\npmatrix{a & 2 \\ 2 &c}$$
with integers $c \geq a \geq 2$. The decomposition 
$$B=\npmatrix{a-2 & 0 \\ 0 & c-2}+2 \npmatrix{1 & 1 \\ 1 & 1}$$
shows that the integer cp-rank of $B$ is at most $4+3+1=8$, unless both $a-2$ and $c-2$ are of the form \eqref{eq:four}. But if $a-2$ is of the form \eqref{eq:four} $a-4$ is not, by Lemma \ref{lem:three}. In this case the decomposition 
$$B=\npmatrix{a-4 & 0 \\ 0 & c-1}+\npmatrix{2 \\ 1}\npmatrix{2 &1}$$
shows that the integer cp-rank of $B$ is at most $3+4+1=8$. We conclude that the integer cp-rank of all such $B$ is at most $8$. 
\end{example} 

\bibliographystyle{plain}
\bibliography{../Bib/CP}


\end{document}